\documentclass[aos]{imsart}
\usepackage{amsmath}
\usepackage{tikz}
\usetikzlibrary{matrix,arrows}
\usepackage{amsfonts,amssymb,amsthm, txfonts, pxfonts,amscd}
\usepackage{graphicx}
\RequirePackage[OT1]{fontenc}
\usepackage{amsthm,amsmath}
\def\Nb{\mathop{\mathbb{N}_{}}\nolimits}

\def\graphsn{\mathop{\mathcal{G}_{[n]}}\nolimits}

\def\graphsU{\mathop{\mathcal{G}_{\mathcal{U}}}\nolimits}

\def\graphsN{\mathop{\mathcal{G}_{\mathbb{N}}}\nolimits}

\def\part{\mathop{\mbox{Part}}\nolimits}

\def\ell{\mathop{[l]}\nolimits}
\def\equalinlaw{\mathop{=_{\mathcal{D}}}\nolimits}

\def\P{\mathop{\mathcal{P}_{}}\nolimits}

\def\0{\mathop{\mathbf{0}_{}}\nolimits}
\def\1{\mathop{\mathbf{1}_{}}\nolimits}

\def\deg{\mathop{\text{deg}}\nolimits}
\def\Rn{\mathop{\mathbf{S}_n}\nolimits}
\def\Rninv{\mathop{\mathbf{S}_n^{-1}}\nolimits}
\def\graphsm{\mathop{\mathcal{G}_{[m]}}\nolimits}
\def\Rmn{\mathop{\mathbf{S}_{m,n}}\nolimits}
\def\Rmninv{\mathop{\mathbf{S}_{m,n}^{-1}}\nolimits}

\newtheorem{thm}{Theorem}[section]
\newtheorem{lemma}[thm]{Lemma}

\newtheorem{defn}[thm]{Definition}
\newtheorem{example}[thm]{Example}

\newtheorem{observation}[thm]{Observation}
\newtheorem{rmk}[thm]{Remark}%\endlocaldefs

\begin{document}
\begin{frontmatter}
%\title{Sparse graph sampling}
%\title{Sparse network modeling}
%\title{Statistical network models}
\title{A framework for statistical network modeling}
\runtitle{Statistical network modeling}

\begin{aug}
\author{\fnms{Harry} \snm{Crane}\ead[label=e1]{hcrane@stat.rutgers.edu} and \fnms{Walter} \snm{Dempsey\ead[label=e1]{wdempsey@uchicago.edu}}}
\thankstext{t1}{Harry Crane is partially supported by NSF CAREER grant DMS-1554092.}
\affiliation{Rutgers University and University of Michigan}
\address{Department
of Statistics \& Biostatistics \\ 
110 Frelinghuysen Road\\
Piscataway, NJ 08854, USA}
\address{Department of Statistics\\
1085 S.\ University Avenue\\
Ann Arbor, MI 48109, USA}

\runauthor{Harry Crane \& Walter Dempsey}
\end{aug}

\begin{abstract}
Basic principles of statistical inference are commonly violated in network data analysis.
Under the current approach, it is often impossible to identify a model that accommodates known empirical behaviors, possesses crucial inferential properties, and accurately 
models the data generating process. In the absence of one or more of these properties, sensible inference from network data cannot be assured.

Our proposed framework decomposes every network model into a {\em (relatively) exchangeable data generating process} and a {\em sampling mechanism} that relates observed data to the population network. 
This framework, which encompasses all models in current use as well as many new models, such as edge exchangeable and relationally exchangeable models, that lie outside the existing paradigm, offers a sound context within which to develop theory and methods for network analysis.
\end{abstract}

\begin{keyword}
\kwd{network data}
\kwd{sparse network}
\kwd{scale-free network}
\kwd{edge exchangeable network}
\kwd{relational exchangeability}
\kwd{relative exchangeability}
\kwd{data generating process}
\kwd{network sampling}
\end{keyword}
\end{frontmatter}

\section{Introduction}\label{section:introduction}

A {\em statistical model}, traditionally defined \cite{CoxHinkley1974,Lehmann1983,McCullagh2002}, is a family of probability distributions $\mathcal{M}$ on the sample space $\mathcal{S}$ of all maps from the set of {\em statistical units} $\mathcal{U}$ into the {\em response space} $\mathcal{R}$.
Some other authors discuss statistical modeling from various perspectives \cite{DrtonSullivant2007,Helland2006,McCullagh2002}, but none of these prior accounts directly addresses the specific challenges of network modeling, namely the effects of sampling on network data and its subsequent impact on inference.
%The recent emergence of nontraditional data structures in biology, physics, engineering, computer science, astronomy, etc.\ has expanded the scope of statistical modeling far beyond its original domain.
%McCullagh \cite{McCullagh2002} provides a category theoretic description of statistical models, while others take a different approach \cite{DrtonSullivant2007,Helland2006}.
%Understandably, these prior accounts do not directly address the specific challenges of network modeling, namely the effects of sampling on network data and its subsequent impact on inference.
In fact, these issues are hardly even mentioned in the statistical literature on networks, a notable exception being the recent analysis of sampling consistency for the exponential random graph model \cite{RinaldoShalizi2013}.
Below we address both logical and practical concerns of statistical inference as well as clarify how specific attributes of network modeling fit within the usual statistical paradigm.
We build up our framework from first principles, which themselves lead to insights that might otherwise pass without notice.

We gear the discussion toward both theorists and practitioners.
For the theorist, we offer a logical framework within which to develop sound theory and methods.
For the practitioner, we present guiding principles for handling network datasets and point out subtle pitfalls of some existing approaches.

\section{Summary of main discussion}\label{section:summary}

We start with the basic principle that a reliable model should be unfazed by arbitrary decisions such as assignment of labels and sampling design:
\begin{itemize}
\item[(A)] ``The sense of the model and the meaning of the parameter[...]may not be affected by accidental or capricious choices such as sample size or experimental design''  \cite[p.\ 1237]{McCullagh2002}.
\end{itemize}

Respectively, concerns over labeling and sampling relate to the logical properties of label equivariance and consistency under subsampling; see Section \ref{section:model properties} for a thorough discussion.
It is notable that almost every network model in popular use fails to satisfy at least one of these properties or otherwise does not accurately model the data generating process: the preferential attachment \cite{BA1999} and superstar \cite{BhamidiSteele2014} models are not label equivariant; the exponential random graph model generally fails to be consistent under subsampling; and the Erd\H{o}s--R\'enyi model does not adequately capture basic network properties.
Some more recent models, such as random graphs built from exchangeable random measures \cite{BorgsChayes2016,CaronFox2014,VeitchRoy2016} and edge exchangeable models \cite{CraneDempsey2016e2}, do address these basic concerns, but nevertheless cannot yet address all relevant aspects of network data.
Despite their misgivings, these models serve a clear practical purpose, and the discussion aims to bridge the divide between standard practice and logical principle.

\subsection{Main considerations}

Invariance principles, and in particular exchangeability, have played an important role in the foundations of statistical and inductive inference since at least the time of de Finetti.
On the one hand, symmetry properties are expedient, and in many ways necessary, for making inferences tractable.
On the other hand, they must be balanced so as to respect important asymmetries in the data.

This balance is especially prominent in modeling network datasets, which often exhibit heterogeneities that are of primary scientific interest.
To be clear, specifying an exchangeable network model poses no technical difficulty---exchangeable random graphs are completely characterized by the Aldous--Hoover theory for partially exchangeable random arrays \cite{Aldous1981,Hoover1979}---but trouble arises when attempting to reconcile exchangeability with the empirical property of sparsity, which reflects the common observation that almost all real world networks have a small number of edges relative to the number of vertices.
Though well known in the probability and combinatorics literature, the following observation has recently been highlighted in the machine learning literature \cite{OrbanzRoy2014}.

\begin{observation}\label{obs:dense or empty}
An exchangeable network is both sparse and nonempty with probability 0.
\end{observation}

%Observation \ref{obs:dense or empty} makes plain the tension between Principle (A) and widely held beliefs about real world networks. 
Even with recent progress in statistical network analysis, including \cite{BickelChenLevina2011,GaoLuZhou2015,ZhangZhou2015,ZhaoLevinaZhu2011,ZhaoLevinaZhuconsistency2011} and many more, there remains no settled approach to address the following basic questions:
\begin{itemize}
	\item[(I)] How can valid inferences be drawn from network models that are not exchangeable and/or sampling consistent in the traditional sense?
	\item[(II)] How can empirical network properties, such as sparsity, power law degree distribution, transitive closure, etc., be modeled in accordance with Principle (A)?
\end{itemize}

With these questions in mind, we lay out a network modeling framework that incorporates existing principles of statistical modeling and brings forward several new ideas relevant to network data.
%We formally define several relevant concepts, such as relative exchangeability, label equivariance, and edge exchangeability, in due course.
Whereas some existing models, such as graphon models and edge exchangeable models, fit easily within our framework, others, such as preferential attachment and exponential random graph models, do not.
Our goal is to nevertheless explain the salient features of these models with respect to the framework, and in the process clarify the potential benefits and drawbacks of each approach.
We highlight four major consequences here; see Section \ref{section:consequences} for a detailed discussion of each.

\begin{enumerate}
	\item[(M1)] {\em Statistical units}.  Recognizing that previous authors, with little exception, either implicitly or explicitly treat vertices as units, we call attention to the possibility that other entities, such as edges, paths, triangles, etc., may act as the units in certain network datasets.  
	%Such is the case when network data are generated by a process of interactions within a population.
This observation bears directly on the sampling mechanism of the two-stage modeling framework in (M2) below and, in some cases, leads to a more natural way to model certain network datasets which addresses Question (II).
% without falling prey to the outcome in Observation \ref{obs:dense or empty}.
	\item[(M2)] {\em Network modeling framework}. We address Question (I), in part, by showing that every label equivariant statistical network model can be specified by
\begin{itemize}
	\item[(i)] a {\em (relatively) exchangeable network generating model}, which models formation of the population network, and
	\item[(ii)] a {\em sampling mechanism}, by which the observed network is obtained by sampling from the population network.
\end{itemize}
In discussing (i), we separate label equivariant models into the two cases of exchangeable and relatively exchangeable network models.
In light of (M1), we discuss notions of exchangeability with respect to relabeling of the units, which may be vertices, edges, paths, or some other component of the network.
When edges are units, for example, {\em edge exchangeable} network models may be appropriate for addressing Question (II) in (M3).
Edge exchangeable models were introduced and developed in \cite{CraneDempsey2016e2} for modeling interaction networks, for which the edges are the fundamental statistical units.
When paths or other entities are the units, the more general notion of {\em relational exchangeability} \cite{CraneDempsey2016relational} may be appropriate. 
Relative exchangeability, on the other hand, is appropriate when modeling data from an inhomogeneous population of vertices, as in community detection.
The framework in (i)-(ii) is made precise in Theorem \ref{thm:main}.
	\item[(M3)] {\em Modeling empirical properties}. Observation \ref{obs:dense or empty} lays bare a severe limitation when network data are regarded as a graph with labeled vertices.   Combining (M1) and (M2), we discuss certain classes of models that reproduce key empirical features while exhibiting tractable invariance properties.
	\item[(M4)] {\em Relative exchangeability}.  Though our proof of Theorem \ref{thm:main} shows that every network can be modeled by an exchangeable network generating process, the exchangeability assumption may be inappropriate for network data from inhomogeneous populations.
	 Relative exchangeability allows us to step outside the boundaries of exchangeable models without sacrificing inferential validity.
	We characterize a large class of relatively exchangeable network models in Theorem \ref{thm:rel exch}.
	\end{enumerate}

The above four points highlight the most crucial considerations if sensible and valid inferences are desired.
As a practical matter, the sampling mechanisms in (M2)(ii) provide the necessary link between population network and observed network data.  
The crux of Theorem \ref{thm:main} is that the much neglected element of sampling is, in fact, implicit in every network model.  Whether chosen finite sample models reflect a realistic data generating process and sampling mechanism ought not be ignored during model selection.

%\subsection{Outline}
%We organize the discussion as follows.
%In Section \ref{section:network models}, we review the fundamentals of network modeling.
%In Section \ref{section:framework}, we summarize our main theorems and introduce our network modeling framework.
%In Section \ref{section:consequences}, we discuss the major consequences of our framework in detail and provide examples.
%In Section \ref{section:inference}, we discuss issues related to inference from network models.
%In Section \ref{section:concluding remarks}, we offer some closing remarks and discussion of future directions.
%In Appendix \ref{appendix:ultra}, we provide some technical definitions related to relative exchangeability.
%In Appendix \ref{appendix:proofs}, we prove our main theorems.

\section{Network modeling}\label{section:network models}

The relevance of Questions (I) and (II) and the associated challenges to statistical inference grow out of the empirical findings of the late 1990s and early 2000s, when several groups recognized common structural features in network datasets from the World Wide Web \cite{BA1999,FFF1999,Kumar1999}, telecommunications systems \cite{Abello1998}, and biological processes \cite{JeongMason2001}.
These observed networks are {\em sparse}, that is, have a small number of edges relative to the number of vertices, and in many cases exhibit a {\em power law degree distribution}, that is, the proportion of vertices with degree $k\geq1$ is asymptotically proportional to $k^{-\gamma}$ for some $\gamma>1$ for all large $k\geq1$.

Barab\'asi \& Albert \cite{BA1999}, cf.\ \cite{Price1965,Simon1955}, propose a preferential attachment mechanism for generating networks with certain power law behavior.
While some, including Barab\'asi \& Albert \cite{BA1999} and D'Souza, et al \cite{DSouza2007}, credit preferential attachment dynamics for the emergence of scale-free network structure, others \cite{AchlioptasClauset2005,AhmedNevilleKompella2010,LeeKimJeong2006,StumpfWiufMay2005,WillingerAldersonDoyle2009}
%notably Achlioptas, et al \cite{AchlioptasClauset2005}, Ahmed, et al \cite{AhmedNevilleKompella2010}, Lee, Kim \& Jeong \cite{LeeKimJeong2006}, Stumpf, et al \cite{StumpfWiufMay2005}, and Willinger, Alderson \& Doyle \cite{WillingerAldersonDoyle2009}, 
point out that common sampling methods can produce network data with vastly different structure than the population network.
All of these latter observations, which highlight the pitfalls of ignoring the effect of sampling on observed network structure, appear outside of the statistics literature.

\subsection{Network data}\label{section:statistical network models}

Intuitively, network data correspond to a graphical structure, as in Figure \ref{fig:labeled}.
Formally, we define {\em data} as a function from the set of {\em statistical units} $\mathcal{U}$ into the {\em response space} $\mathcal{R}$ and we define a {\em statistical model} as a set of probability distributions on the space of functions $\mathcal{U}\rightarrow\mathcal{R}$.
Two seemingly elementary questions of statistical modeling, then, are
\begin{itemize}
	\item What are the units?
	\item What is the response space?
\end{itemize}	

We find no discussion of these questions in the networks literature aside from a comment by Kolaczyk \cite[p.\ 54]{Kolaczykbook}, who mostly focuses on the case in which vertices are the units but raises some other possibilities in the course of his discussion.
This convention is taken for granted in other places, where without exception {\em network data} is regarded as synonymous with a {\em graph} $G=(S,E)$ with vertex set $S$ and edges $E\subseteq S\times S$.\footnote{In some cases, network data include multiple edges or edges involving more than two vertices.
For the sake of clarity, we omit these extensions and treat all edges as undirected, that is, $(i,j)\in E$ implies $(j,i)\in E$ so that we may write $ij\in E$.}
As we discuss, the seemingly innocuous act of identifying the vertices as units causes much of the confusion in network modeling.
In fact, the vertices, edges, or various other components of a network may be natural candidates for the units depending on the application at hand.
More than just an technical observation, the identification of the units has implications for how the data ought to be represented and ultimately modeled.

We imagine a population $V$ and we write $\mathcal{U}$ as the set of units.
We may reasonably identify the vertices as units, and write $\mathcal{U}=V$, if the data are obtained by sampling $S\subset V$ and observing a network of binary relations among sampled individuals.
Alternatively, many network datasets arise by observing interactions between individuals in the population $V$, as in networks formed by professional collaborations \cite{BA1999,Rossi2015} and email communications \cite{KlimtYang2004}.
Network data then consist of the sampled interactions (as edges) along with whatever vertices are involved in the observed interactions.
The edges may be treated as units in this case.
We label network data according to which entities comprise the units, as in Figures \ref{fig:labeled}(a) and \ref{fig:labeled}(d).
Still other entities, such as paths, may act as the natural statistical unit, as in the mapping of the Internet via path sampling.

The framework below is meant to treat network data of a generic form, but to aid clarity we focus mostly on the case in which vertices or edges are units, with some discussion of paths as units peppered throughout.
For all intents and purposes, the intuitive visualization of network data as a vertex or edge labeled graph, as in Figures \ref{fig:labeled}(a) and \ref{fig:labeled}(d), is sufficient to follow along.
The following definition gives meaning to the more technical aspects of our discussion and allows us to speak of network data generically, with possibly different interpretations given to the symbols as appropriate.

\begin{defn}[Network data]\label{defn:data}
Let $V$ be a population.
{\em Network data} for a population $V$ is a function $G:\mathcal{U}\to\mathcal{R}$, where $\mathcal{U}$ is the set of units and $\mathcal{R}$ is the response space.
The interpretation of $\mathcal{U}$, $\mathcal{R}$, and $G:\mathcal{U}\to\mathcal{R}$ depends on whether vertices or edges act as units:
\begin{itemize}
	\item {\em vertices}: $\mathcal{U}=V$ and $\mathcal{R}=2^{\mathcal{U}}$, the power set of all subsets of $\mathcal{U}$, so that, for each $u\in\mathcal{U}$, $G(u)\subseteq\mathcal{U}$ is the set of all $u'\in\mathcal{U}$ for which there is an edge between $u$ and $u'$.
This case is often represented by a graph $(\mathcal{U},E)$ with edge set $E\subseteq\mathcal{U}\times\mathcal{U}$ defined by
\[uu'\in E\quad\text{if and only if}\quad u'\in G(u);\]
 see Figure \ref{fig:labeled}(a).
	\item {\em edges}: $\mathcal{U}$ corresponds to interactions and $\mathcal{R}=(V\times V)/\mathcal{S}_2$ consists of unordered pairs of elements in the population so that $G:\mathcal{U}\to (V\times V)/\mathcal{S}_2$ identifies which vertices $v,v'\in V$ are involved in the interaction corresponding to unit $u\in\mathcal{U}$.\footnote{$\mathcal{S}_2$ is the set of permutations of 2 elements so that $(V\times V)/\mathcal{S}_2$ is the quotient space of ordered sets $(v,v')$ with order ignored.}  This case also corresponds to a graph, but with labeled edges instead of vertices; see Figure \ref{fig:labeled}(d).
	\item {\em paths}: $\mathcal{U}$ corresponds to paths between vertices and $\mathcal{R}=\bigcup_{n\geq1}V^n$ consists of finite tuples of vertices so that $G:\mathcal{U}\to\mathcal{R}$ identifies the sequence of vertices $(v_1,\ldots,v_n)\in V^n$ involved in each path.
\end{itemize}
\end{defn}

\begin{figure}[!t]
\includegraphics[width = 0.5\textwidth]{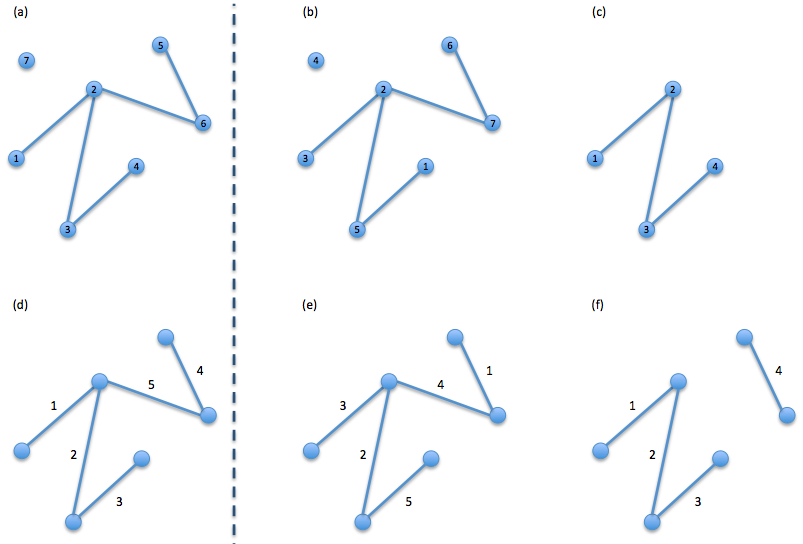}
\caption{Two views of network data: Panels (a)-(c) show network data with labeled vertices, as is appropriate when vertices are units.  Panels (d)-(f) show network data with labeled edges, as is appropriate when edges are units.  Panel (b) shows the network from (a) relabeled according to permutation $(135674)(2)$.  Panel (e) shows the network from (d) relabeled according to permutation $(1354)(2)$.  Panel (c), respectively (f), shows the restriction of the network in (a), respectively (d), to the units labeled in $\{1,2,3,4\}$.}
\label{fig:labeled}
\end{figure}

Unless otherwise noted, we assume a countable collection of units $\mathcal{U}=\Nb=\{1,2,\ldots\}$.
We always write $G$ to denote network data, with the understanding that $G$ is labeled according to which of its components are the units.
%Specifically, the discussion below specializes to the case in which vertices or edges are the units.
Below we write $\mathcal{G}_S$ to denote the set of graphs with units labeled in $S\subset\Nb$, which are mostly interpreted as vertex or edge labeled graphs in the context below.
We then write $\mathcal{P}(\mathcal{G}_S)$ to denote the set of probability distributions on $\mathcal{G}_S$ equipped with its Borel $\sigma$-field.

\begin{defn}[Network model]\label{defn:model}
A {\em network model} for population $V$ and units $\mathcal{U}$ is a subset $\mathcal{M}\subseteq\mathcal{P}(\graphsU)$.
A {\em parameterized network model} is a {\em parameter set} $\Theta$ together with a map $P:\Theta\rightarrow\P(\graphsU)$ that associates each $\theta\in\Theta$ with a probability distribution $P_{\theta}$ on $\graphsU$.

We call $\mathcal{M}\subseteq\mathcal{P}(\graphsU)$ a {\em data generating model} to emphasize that $\mathcal{M}$ comprises a family of data generating processes for the population network.
For $S\subset\mathcal{U}$ with $|S|<\infty$, we call $\mathcal{M}_S\subseteq\P(\mathcal{G}_S)$ a {\em finite sample model}.
\end{defn}

\subsection{Sampling}

Many of the fundamental issues in network modeling are caused by a logical disconnect between the finite sample models $\{\mathcal{M}_n\}_{n\geq1}$ and a presumed data generating model $\mathcal{M}\subseteq\P(\graphsN)$.
Theorem \ref{thm:main} connects the two by a sampling procedure.

For $n\geq m\geq1$, we define the {\em canonical sampling operations} $\Rn:\graphsN\rightarrow\graphsn$ and $\Rmn:\graphsn\rightarrow\graphsm$ by $G\mapsto \Rn G:=G_{|[n]}$ and $G\mapsto\Rmn G:=G_{|[m]}$, respectively.
Every measure $\mu$ on $\graphsN$ induces a measure on $\graphsn$ in the usual way, $\mu\mapsto\mu\Rninv=:\Rn\mu$.
%For $\mathcal{M}_m\subseteq \P(\graphsm)$ and $\mathcal{M}_n\subseteq\P(\graphsn)$, with $n\geq m\geq1$, we write $\mathcal{M}_m\subseteq\mathcal{M}_n$ to denote that $\mathcal{M}_m$ is a {\em submodel} of $\mathcal{M}_n$, that is, (i) to every $\mu_m\in\mathcal{M}_m$ there exists $\mu_n\in\mathcal{M}_n$ such that $\Rmn\mu_n=\mu_m$.
%We write $\mathcal{M}_m=\Rmn\mathcal{M}_n$ if, in addition to (i), (ii) $\Rmn\mu_n\in\mathcal{M}_m$ for all $\mu_n\in\mathcal{M}_n$.

In practice, network models are often specified through their {\em finite sample models} $\{\mathcal{M}_n\}_{n\geq1}$, where each $\mathcal{M}_n\subseteq\P(\graphsn)$ indicates the model for data from a sample of $n=1,2,\ldots$ units.
In the parametric setting, we specify a sequence of finite sample models $\{\mathcal{M}_n\}_{n\geq1}$ through a collection of maps $P^{(n)}:\Theta\rightarrow\P(\graphsn)$, so that $\mathcal{M}_n=P^{(n)}\Theta$ for each $n=1,2,\ldots$.

A data generating model $\mathcal{M}\subseteq\P(\graphsN)$ is {\em finitely specified} when it is defined through a collection $\{\mathcal{M}_n\}_{n\geq1}$ of consistent finite sample models.
Any $\mathcal{M}\subseteq\P(\graphsN)$ can be finitely specified by taking 
\begin{equation}\label{eq:finitely specified}
\mathcal{M}_n=\Rn\mathcal{M}:=\{\mu\Rninv:\,\mu\in\mathcal{M}\} \quad\text{for each }n\geq1,\end{equation}
but these induced finite sample models may not accurately reflect the relationship between population network and observed subnetwork data.
In particular, the canonical sampling maps $\{\Rn\}_{n\geq1}$ may not properly model the sampling mechanism \cite{AchlioptasClauset2005,LeeKimJeong2006,WillingerAldersonDoyle2009}.

Under the canonical sampling operation, $\{P^{(n)}\Theta\}_{n\geq1}$ corresponds to a data generating model $P:\Theta\rightarrow\P(\graphsN)$ only if the families $P^{(n)}\Theta$ are {\em consistent under subsampling}, that is, $P^{(n)}\Theta=\Rn P\Theta$ for every $n\geq1$.
Quoting Shalizi \& Rinaldo \cite[p.\ 510]{RinaldoShalizi2013},
\begin{quote}
``When this form of consistency fails, then the parameter estimates obtained from a sub-network may not provide reliable estimates of, or may not even be relatable to, the parameters of the whole network, rendering the task of statistical inference based on a sub-network ill-posed.'' 
\end{quote}
Nevertheless, many widely used network models are not consistent with respect to the canonical sampling operation.

\begin{example}\label{ex:ERGM}
The {\em exponential random graph model} (ERGM) \cite{HollandLeinhardt1981} is an exponential family of distributions for vertex labeled networks of a given finite size $n=1,2,\ldots$.
Let $T=(T_1,\ldots,T_k)$ be a collection of network statistics and $\theta=(\theta_1,\ldots,\theta_k)\in\Theta$ be parameters.
The ERGM $P^{(n)}:\Theta\to\mathcal{P}(\graphsn)$ with natural parameter $\theta$ and canonical sufficient statistic $T$ on $\graphsn$ assigns probabilities
\begin{equation}\label{eq:ERGM}
P^{(n)}_{\theta}(G)\propto\exp\left\{\sum_{i=1}^k \theta_iT_i(G)\right\},\quad G\in\graphsn.
\end{equation}

The finite sample models determined by \eqref{eq:ERGM} are consistent under subsampling only under the restrictive condition that the sufficient statistics have {\em separable increments} \cite{RinaldoShalizi2013}.
For example, let $\beta=(\beta_1,\ldots,\beta_n)$ and construct a random graph $G=([n],E)$ so that each edge is present independently with probability
\[\mathbb{P}\{ij\in E\}= \frac{e^{\beta_i + \beta_j}}{1+e^{\beta_i + \beta_j}},\quad 1\leq i<j\leq n.\]
The Erd\H{o}s--R\'enyi distribution corresponds to the case $\beta_i\equiv \beta$ for all $i=1,\ldots,n$.
%  This model, in general, is label equivariant but not exchangeable.  
%The expected degree for 
%vertex~$i$ is~$d^{(n)}_i = \sum_{j \in [n] \backslash \{ i \} } p_{ij}$. We discuss the $\beta$-process
%further in Section~\ref{section:power law}.
\end{example}

\begin{example}\label{ex:Bickel}
The following is a special case of the models proposed in \cite{BickelChen2009PNAS}.
For $\Theta=[0,1]$ and $n\geq1$, we define $P^{(n)}:[0,1]\to\mathcal{P}(\graphsn)$ as the model that associates each $\theta\in[0,1]$ to the Erd\H{o}s--R\'enyi distribution with parameter $\theta/n$, that is, the distribution of a vertex labeled graph with $n$ vertices for which each edge is present independently with probability $\theta/n$.
For $m\leq n$, $P^{(m)}_{\theta}$ and $P^{(n)}_{\theta}$ are not consistent with respect to the canonical subsampling map $\Rmn$.
Without further information relating samples of different sizes, inferences cannot extend beyond the observed network.
We discuss this further in Sections \ref{section:sampling} and \ref{section:infer-univ}.
\end{example}

The canonical sampling mechanism is unrealistic for most applications.
From this perspective, the prevalence of network models that are not consistent with respect to the canonical sampling maps, as in Examples \ref{ex:ERGM} and \ref{ex:Bickel}, may not pose a fundamental inferential problem, as long as the inconsistency between finite sample models of different size can be explained.

Inductive and/or predictive inferences based on $\{\mathcal{M}_n\}_{n\geq1}$ remain possible if $\mathcal{M}_n$ and $\mathcal{M}_m$ can be related through some subsampling mechanism.
To treat this,  we let $\Sigma_n:\graphsN\rightarrow\graphsn$ and $\Sigma_{m,n}:\graphsn\rightarrow\graphsm$, $m\leq n$, be (possibly random) sampling maps.
Given $\mu\in\P(\graphsN)$, we write $\Sigma_n\mu\in\P(\graphsn)$ to denote the distribution of $\Sigma_n G$ for $G\sim\mu$, and likewise for $\mu_n\in\P(\graphsn)$ and $\Sigma_{m,n}\mu_n\in\P(\graphsm)$.
For a data generating model $\mathcal{M}\subseteq\P(\graphsN)$ and sampling mechanism $\Sigma_n:\graphsN\to\graphsn$, we write $\Sigma_n\mathcal{M}=\{\Sigma_n\mu:\,\mu\in\mathcal{M}\}$.

\begin{example}[Snowball sampling]\label{example:snowball}
In snowball sampling, we start with a single vertex $v^*$ and sample outwardly by first choosing all vertices connected to $v^*$, then choosing all vertices connected to the vertices chosen in the last step, and so on until a prescribed number of vertices is sampled.\footnote{There are variations of snowball sampling where at each step at most $k\geq1$ neighbors are sampled.  The details are not crucial here.}
Since the sampled network is highly dependent on the initially chosen vertex $v^*$, the characteristics of the sampled network may not be representative of the population network.
Lee, Kim \& Jeong \cite{LeeKimJeong2006} analyze the discrepancy between various sample and population statistics under snowball sampling from the preferential attachment model.

Snowball sampling determines a sampling mechanism $\{\Sigma_n\}_{n\geq1}$ such that $\Sigma_n\mu$ and $\Sigma_m\mu$ may not be consistent under subsampling.
Nevertheless snowball sampling, including the closely related respondent driven sampling \cite{Heckathorn1997,Wejnert2010}, is widely used in practice, and so it is fitting to establish a context for network modeling that allows for snowball sampling and other common sampling mechanisms.
\end{example}

%Given a sample of $n\geq1$ units $[n]:=\{1,\ldots,n\}$, we write $G_{|[n]}$ to denote the {\em subnetwork}, or {\em sampled data on $[n]$}, given by removing any units not labeled in $[n]$.
%Assuming $G_n$ is network data for a sample of $n$ units, %By Definition \ref{defn:data}, these actions are best understood as the usual restriction and relabeling operations on graphs, as demonstrated in Figure \ref{fig:labeled}.

\section{Modeling framework}\label{section:framework}

The above discussion suggests that every {\em bona fide} network model should account for both the data generating process and the sampling mechanism.
The accounting can be implicit, as when the model is specified by a family of finite sample models $\{\mathcal{M}_n\}_{n\geq1}$, or explicit, as when the data generating process and sampling scheme are modeled directly.
Implicit modeling may be justified if the law governing observed data is well understood, but in practice this is quite rare.

\begin{defn}[Statistical network model]\label{defn:statistical network model}
A {\em statistical network model} is a pair $(\mathcal{M},\{\Sigma_n\}_{n\geq1})$, where $\mathcal{M}\subseteq\P(\graphsN)$ models the data generating process and each $\Sigma_n:\graphsN\to\graphsn$ is a (possibly random) sampling mechanism such that $\Sigma_n\mathcal{M}$ models data for a sample of size $n\geq1$.
A {\em parameterized model} is a triple $(\Theta,Q,\{\Sigma_n\}_{n\geq1})$, from which the data generating process is modeled by $Q:\Theta\to\P(\graphsN)$.
\end{defn}

\begin{rmk}
For clarity, we focus on the parametric case below.
\end{rmk}

Our main theorem establishes that no generality is lost by modeling network data as in Definition \ref{defn:statistical network model}.
We state the theorem now and explain its significance afterward. 

\begin{thm}\label{thm:main}
Let $\Theta$ be a parameter space that satisfies Condition \ref{condition:uncountable} and let $\{P^{(n)}\}_{n\geq1}$ define label equivariant finite sample models $P^{(n)}:\Theta\to\mathcal{P}(\graphsn)$.
Then there exists an identifiable, (relatively) exchangeable data generating model $Q:\Theta\to\mathcal{P}(\graphsN)$ and (possibly random) sampling mechanisms $\{\Sigma_n\}_{n\geq1}$ such that $G_n\sim P^{(n)}_{\theta}$ satisfies $G_n\equalinlaw\Sigma_n G$ for $G\sim Q_{\theta}$, where $\equalinlaw$ denotes {\em equality in law}.
\end{thm}

Though our proof of Theorem~\ref{thm:main} shows that $Q$ can always be chosen to be exchangeable, the theorem allows the data generating model to be relatively exchangeable.
The appropriateness of an exchangeable or relatively exchangeable data generating model depends on context.  
A relatively exchangeable model is most appropriate when the parameter space incorporates inhomogeneity, as in models for community detection.  When the population is homogeneous, however, an exchangeable data generating model often yields the best interpretation.

\subsection{Model properties and inference}\label{section:model properties}

We suffer no loss of generality in assuming an {\em identifiable} data generating model $Q:\Theta\to\mathcal{P}(\graphsN)$, that is, $\theta\neq\theta'$ implies $Q_{\theta}\neq Q_{\theta'}$.

\paragraph{Label equivariance}

Principle (A) says that any statistical model should be robust to arbitrary manipulations of the data, among other things insisting that the parameter's {\em meaning} should not change under relabeling of the data.
This should not be read as a mandate that every element of the model is invariant with respect to relabeling, but rather that the model itself does not depend on the chosen labeling.

\begin{defn}[Label equivariance]\label{defn:label equivariance}
A network model $(\Theta,Q,\{\Sigma_n\}_{n\geq1})$ is {\em label equivariant} if, for every permutation $\sigma:\Nb\rightarrow\Nb$ and every $\theta\in\Theta$, there exists $\theta'\in\Theta$ such that $G\sim Q_{\theta}$ implies $G^{\sigma}\sim Q_{\theta'}$.
In this case, every $\sigma:\Nb\rightarrow\Nb$ defines an action on $\Theta$ by $\sigma\theta=\theta'$ and $\sigma\Theta=\Theta$ for all permutations $\sigma:\Nb\rightarrow\Nb$.
\end{defn}

Label equivariance captures the basic condition that if $\mathcal{M}=(\Theta,Q,\{\Sigma_n\}_{n\geq1})$ models network data $G_n$, then $\mathcal{M}$ remains the model if $G_n$ is relabeled to $G_n^{\sigma}$ for any permutation $\sigma:[n]\to[n]$, where $G_n^{\sigma}$ is the network data after relabeling the units by permutation $\sigma:[n]\to[n]$.  
Figures \ref{fig:labeled}(c) and \ref{fig:labeled}(f) show the action of relabeling for vertex and edge labeled graphs, respectively.

%Label equivariance should not be conflated with the much stronger condition of exchangeability.
Label equivariance should not be confused with exchangeability: whereas exchangeability is a distributional property, equivariance is a model property.
A probability distribution $P_{\theta}$ on $\mathcal{G}_S$ is {\em exchangeable} if $G\sim P_{\theta}$ implies $G^{\sigma}\equalinlaw G$ for all permutations $\sigma: S\to S$.
Nevertheless, we borrow the terminology and call $(\Theta,Q,\{\Sigma_n\}_{n\geq1})$ an {\em exchangeable model} if every $\sigma:\Nb\to\Nb$ acts trivially on $\Theta$, that is, $\sigma\theta=\theta$ for all $\theta\in\Theta$.
With this understanding, we observe that any exchangeable model is label equivariant, but not every label equivariant model is exchangeable; take the stochastic blockmodel (Example \ref{example:SBM}) for instance.
%Some popular models, such as the preferential attachment model, are not even label equivariant. 

\paragraph{Inference problems}

Additional model properties are reasonable when we consider the following two inference problems.  We discuss each in turn.
\begin{itemize}
	\item {\em Universal parameters} for the population network are estimated based on a sampled subnetwork.
A common example includes estimating the power law exponent \cite{BA1999,FFF1999}.
	\item {\em Latent structure} of sampled units is inferred based on structural properties of the observed network.
Community detection \cite{ZhaoLevinaZhuconsistency2011} is a primary example, as is missing link estimation \cite{ClausetMooreNewman2008,Kossinets2006}.
\end{itemize}

\subsubsection{Inferring universal parameters}\label{section:universal}

When interested in universal parameters, that is, properties of the population network, the parameter space $\Theta$ in $(\Theta,Q,\{\Sigma_n\}_{n\geq1})$ must be preserved under the operations of labeling and sampling.
In addition to identifiability and label equivariance of the data generating model $Q:\Theta\to\P(\graphsN)$, the action of the sampling mechanism must respect the furnishings of $\Theta$, as reflected in the property of complete identifiability.

\paragraph{Complete identifiability}
For every $n\in\Nb$ and a (possibly random) sampling map $\Sigma_n:\graphsN\rightarrow\graphsn$, we define an equivalence relation $\sim_{\Sigma_n}$ on $\Theta$ by $\theta\sim_{\Sigma_n}\theta'$ if and only if $G\sim P_{\theta}$ and $G'\sim P_{\theta'}$  implies $\Sigma_nG\equalinlaw\Sigma_nG'$.
With $\Theta_n:=\Theta/\sim_{\Sigma_n}$ as the quotient space, we write $[\theta]_{\Sigma_n}\in\Theta_n$ to denote the equivalence class of $\theta$ under $\sim_{\Sigma_n}$.

For each $n\in\Nb$, let us define $Q^{(n)}:\Theta\rightarrow\P(\graphsn)$ as the model governing $\Sigma_nG$, where $G$ is modeled by $Q:\Theta\rightarrow\P(\graphsN)$.
By the above equivalence relation $\sim_{\Sigma_n}$, these finite sample models need not be identifiable since $\theta\sim_{\Sigma_n}\theta'$ implies $Q^{(n)}_{\theta}=Q^{(n)}_{\theta'}$.
Thus, observed network data for sampled units $[n]\subset\Nb$ is only valid for inference of the equivalence classes of $\Theta_n$.
To emphasize this, we may write $Q^{(n)}:\Theta_n\rightarrow\P(\graphsn)$.

\begin{defn}[Complete identifiability]\label{defn:complete identifiability}
A  network model $(\Theta,Q,\{\Sigma_n\})_{n\geq1})$ is {\em completely identifiable} if $\Theta/\sim_{\Sigma_n}\cong\Theta$ for all $n\in\Nb$.
\end{defn}

Many network models are completely identifiable, such as the Erd\H{o}s--R\'enyi \cite{ErdosRenyi1959,ErdosRenyi1960} and preferential attachment models\cite{BA1999,ChungLubook}, while others are not, such as the stochastic blockmodel \cite{HollandLaskeyLeinhardt1983}.
Lacking complete identifiability is not a pathology, as such models typically incorporate structural inhomogeneities and are often geared toward inferring latent structure.

\subsubsection{Inferring latent structure}\label{section:latent}

The above induced actions of relabeling and sampling on $\Theta$ are nontrivial in models for latent structural properties.
Consider the case in which $\Theta$ decomposes as $\Theta=\Phi\times\Psi$ such that every permutation $\sigma:\Nb\to\Nb$ acts by $\sigma(\phi,\psi)=(\phi,\sigma\psi)$.
Specifically, we assume that $\Phi$ does not depend on the labeling and $\Psi$ consists of labeled objects so that $\sigma\psi\neq\psi$ for some $\psi\in\Psi$ and $\sigma:\Nb\to\Nb$.
See Appendix \ref{appendix:ultra} for a more precise definition of the generic term {\em labeled object}. 
Examples \ref{example:SBM} and \ref{example:covariates} illustrate typical cases.

From network data $G_n\in\graphsn$ and a model $(\Theta,Q,\{\Sigma_n\}_{n\geq1})$, we can hope to infer $(\phi,\psi|_{[n]})$, where $\psi|_{[n]}$ is the restriction of $\psi$ to a combinatorial object labeled in $[n]$, only if the model behaves well with respect to the map $(\phi,\psi)\mapsto(\phi,\psi|_{[n]})$.
For such inference label equivariance gives way to the stronger notion of relative exchangeability, which insures that inference for $\psi|_{[n]}$ depends on $G$ only through the observed data $G_n$.

\paragraph{Relative exchangeability}
In experimental design, the {\em lack of interference} assumption assumes that a change in treatments assigned to units $u'\neq u$ does not affect the response for unit $u$; see Cox \cite[Section 2.4]{Cox1958}.
We adapt this notion to inference for latent structure by assuming that the distribution of observed data is unaffected by alterations to unseen parts of the network.

\begin{defn}[Relative exchangeability]\label{defn:relative exchangeability}
Let $\Theta=\Phi\times\Psi$ be as above.  
%A data generating model $Q:\Theta\to\P(\graphsN)$ is {\em exchangeable relative to $\Psi$} if, for all $\theta=(\phi,\psi)\in\Theta$ and all permutations $\sigma:\Nb\to\Nb$ such that $(\phi,\psi|_{[n]})=(\phi,(\sigma\psi)|_{[n]})$, $G\sim Q_{\theta}$ and $G'\sim Q_{\sigma\theta}$ satisfy $\Rn G^{\sigma}\equalinlaw\Rn G'$.
A network model $(\Theta,Q,\{\Sigma_n\}_{n\geq1})$ is {\em exchangeable relative to $\Psi$} if, for all $\theta=(\phi,\psi)\in\Theta$ and permutations $\sigma:\Nb\to\Nb$ such that $(\phi,\psi|_{[n]})=(\phi,(\sigma\psi)|_{[n]})$, $G\sim Q_{\theta}$ satisfies $\Sigma_nG^{\sigma}\equalinlaw\Sigma_nG'$ for $G'\sim Q_{\sigma\theta}$.
\end{defn}

Note carefully what Definition \ref{defn:relative exchangeability} says about the finite sample models imposed by $(\Theta,Q,\{\Sigma_n\}_{n\geq1})$.
%For concreteness, suppose $\Psi$ consists of all equivalence relations $\approx_{\psi}$ on $\Nb$.
Each sampling mechanism $\Sigma_n$ determines an equivalence relation $\sim_{\Sigma_n}$ on $\Psi$ by  $\psi\sim_{\Sigma_n}\psi'$ if and only if $G\sim Q_{\phi,\psi}$ and $G'\sim Q_{\phi,\psi'}$ implies $\Sigma_nG\equalinlaw\Sigma_n G'$ for all  $\phi\in\Phi$; see Section \ref{section:universal} above.
We define $\Psi_n=\Psi/\sim_{\Sigma_n}$ as the associated quotient space.
By the combinatorial structure of $\Psi$, the ordinary restriction $\psi\mapsto\psi|_{[n]}$ also determines an equivalence relation $\sim_n$ on $\Psi$ by $\psi\sim_n\psi'$ if and only if $\psi|_{[n]}=\psi'|_{[n]}$.
Relative exchangeability holds for $(\Theta,Q,\{\Sigma_n\}_{n\geq1})$ if, for any permutation $\sigma:\Nb\to\Nb$ and $\psi\in\Psi$, $\psi\sim_n\sigma\psi$ implies $\psi\sim_{\Sigma_n}\sigma\psi$.
In this way, the law of $\Sigma_nG$ for $G\sim Q_{\phi,\psi}$ depends on $\psi$ only through $\psi|_{[n]}$.

The preceding paragraph is more general than is necessary for most applications.
We streamline the discussion by specializing to the case $\Sigma_n=\Rn$ for every $n\geq1$.
Under this simplification, we get a more palatable interpretation of relative exchangeability as a combination of label equivariance and lack of interference for network data.
In particular, if $G$ is a population network from a relatively exchangeable model, then network data $G_n$ for a sample $[n]\subset\Nb$ is invariant with respect to permutations $\sigma:[n]\to[n]$ that act as the identity on $\Psi/\sim_{\Rn}$, regardless of how the action of $\sigma$ extends outside of $[n]$.
We also note that relatively exchangeable data generating models bequeath label equivariance to their finite sample models.
We characterize relatively exchangeable network models in Theorem \ref{thm:rel exch}.

% via the action of $\{\Sigma_n\}_{n\geq1}$ on $\Theta$.
%Recovering $\psi|_{[n]}$ from data $G_n$ requires the further condition that the quotient spaces $\Psi/\sim_{\Sigma_n}$ and $\Psi/\sim_n$ are in {\em bijective correspondence} -- that is, there exists a bijection on the parameter
%space~$\beta:\Psi\to\Psi$ and bijection between the quotient spaces $\beta_n:\Psi/\sim_{\Sigma_n}\to\Psi/\sim_n$, for each $n\geq1$, so that transformation of $\Psi$ via $\beta$ followed by canonical sampling is equivalent in
%distribution to random sampling via~$\Sigma_n$ followed by mapping via~$\beta_n$ to the alternative quotient space.
% \begin{equation}\label{eq:commute}
% \begin{tikzpicture}
% \matrix (m) [matrix of math vertices, row sep=2em,
% column sep=2em, text height=1.5ex, text depth=0.25ex]
% {\Psi&&\Psi\\
%    & &\\
% \Psi/\sim_{\Sigma_n}&&\Psi/\sim_{n} \\ };
% \path[->]
% (m-1-1) edge[bend left=15] vertex[auto]{$\beta$} (m-1-3)
% (m-1-3) edge[bend left=15] vertex[auto]{$\beta^{-1}$} (m-1-1)
% (m-1-1) edge vertex[swap,auto]{$\sim_{\Sigma_n}$} (m-3-1)
% (m-1-3) edge vertex[auto]{$\sim_n$} (m-3-3)
% (m-3-1) edge[bend left=15] vertex[auto]{$\beta_n$} (m-3-3)
% (m-3-3) edge[bend left=15] vertex[auto]{$\beta_n^{-1}$} (m-3-1);
% \end{tikzpicture}
% \end{equation}

\begin{example}[Stochastic blockmodel]\label{example:SBM}
The stochastic blockmodel \cite{HollandLaskeyLeinhardt1983} is a statistical network model $(\Theta,Q,\{\Rn\}_{n\geq1})$ whose parameter space includes a partition of vertices.
Here the vertices are labeled and, for this example, we specialize to the case $\Theta=[0,1]\times[0,1]\times\mathcal{P}_{\Nb}$, where $\mathcal{P}_{\Nb}$ is the set of partitions of $\Nb$.
In the notation above, we have $\Theta=\Phi\times\Psi$ with $\Phi=[0,1]\times[0,1]$ and $\Psi=\mathcal{P}_{\Nb}$.

The data generating model $Q:\Theta\rightarrow\P(\graphsN)$ assigns each $(p,q)\in\Phi$ and $B\in\Psi$ to the distribution of a random vertex labeled graph $G=(V,E)$ for which each edge is present or absent independently with probability 
\begin{equation}\label{eq:SBM-eq}
\mathbb{P}\{ij\in E\}=\left\{\begin{array}{cc} p,& i\approx_B j,\\
q,& \text{otherwise,}
\end{array}\right.\end{equation}
where $i\approx_B j$ indicates that $i$ and $j$ are in the same cluster of $B$.

Since the status of each edge $ij$ in $G$ depends only on the relationship between $i$ and $j$ in $B$, the marginal distribution of $G_{|\{i,j\}}$ depends only on the restriction $B|_{\{i,j\}}$.
It follows that the distribution of $G_{|[n]}$ depends on $B$ only through $B|_{[n]}$ for every $[n]\subset\Nb$.
Thus, the canonical sampling map $\Rn$ acts on $\Theta$ by $(p,q,B)\mapsto(p,q,B|_{[n]})$, so that $B|_{[n]}=B'|_{[n]}$ implies $G_{|[n]}\equalinlaw G'_{|[n]}$ for $G'$ distributed according to $Q_{\theta}$ with $\theta=(p,q,B')$.
In this case, $\Psi/\sim_{\Rn}$ corresponds to $\mathcal{P}_{[n]}$, the space of partitions of $[n]$.
The model, therefore, is not completely identifiable, but it is label equivariant and relatively exchangeable, and data $G_n\in\graphsn$ for a sample $[n]\subset\Nb$ is sufficient for inference of the subparameter $B|_{[n]}$.
\end{example}

\begin{example}[Models with covariates]\label{example:covariates}
Relative exchangeability is not confined to the inference of latent structure.
It is also appropriate when units come equipped with covariate information.
Let $\Theta=\mathbb{R}^d\times(\mathbb{R}^d)^{\Nb}$ so that $\theta=(\theta,\mathbf{x})$ consists of a parameter $\theta=(\theta_1,\ldots,\theta_d)$ and covariates $\mathbf{x}=(x_i)_{i\in\Nb}$ with $x_i=(x_{i,1},\ldots,x_{i,d})$ for each vertex $i=1,2,\ldots$.
We define a relatively exchangeable model $(\Theta,Q,\{\Rn\}_{n\geq1})$ by assuming $G=(\Nb,E)$ with distribution $Q_{\theta,\mathbf{x}}$ is a vertex labeled graph with each edge present independently with probability
\[\mathbb{P}\{ij\in E\}=\frac{e^{\theta_1(x_{i,1}+x_{j,1})+\cdots+\theta_d(x_{i,d}+x_{j,d})}}{1+e^{\theta_1(x_{i,1}+x_{j,1})+\cdots+\theta_d(x_{i,d}+x_{j,d})}}.\]
Many other network models with covariates can be built out of this example.
The above choice merely demonstrates that such models are relatively exchangeable.
\end{example}

\section{Consequences of the network modeling framework}\label{section:consequences}

Our framework above guides our choice of network model in a way that addresses known empirical properties and ensures valid inferences.

\subsection{Identifying units}

We begin with a plain, albeit crucial, discussion about the statistical units.
Though more often discussed in the context of designed experiments, consideration of units enters into network studies at the sampling stage.
%Here the {\em experimental units} are more aptly called {\em observable units} to reflect the dependence on sampling design.
A few examples illustrate that the edges can and should be treated as the units in many applications.

\paragraph{Karate club dataset}

Zachary's karate club dataset \cite{Zachary1977} records social interactions among 34 members in a karate club.
A network is obtained by putting an edge between two individuals if they interacted at least once outside the karate club.
There is no sampling issue in this case: the social interactions among all 34 club members are observed.
The data can be represented as a graph \cite[Fig.\ 1, p.\ 456]{Zachary1977}  and the vertices may be treated as units.

\paragraph{Actors collaboration network}

The actors collaboration network \cite{BA1999} contains information about a sample of movies and the actors involved.
Associated to each movie is the set of actors in its cast, so that each movie $u$ in which both $v,v'\in V$ act corresponds to an interaction between $v$ and $v'$.
The dataset is obtained by sampling movies and, thus, the units correspond to edges.
Since popular actors tend to act in more movies, a random sample of movies results in a sample of actors that is size biased according to degree.

\paragraph{Enron email network}

The Enron email corpus \cite{KlimtYang2004} contains information about email correspondence within the Enron Corporation.
Each email is sent from one employee to a list of recipients, so that the observed network is obtained by sampling these emails and once again the edges should be treated as units.

\paragraph{Internet topology}

The traceroute approach to sampling the router-level Internet samples paths by sending packets out along the Internet backbone that record the sequence of servers visited in sending a message from one point to another.
The observed network is assembled by piecing together the various paths obtained from this procedure, making the paths a natural unit for this dataset.

\subsection{Model specification}\label{section:specification}

A family of finite sample models determines a well defined data generating model only if models for different sample sizes are consistent under subsampling.
An inconsistent family of finite sample models leaves unspoken the assumption that the data generating mechanism varies with the choice of sample size, in direct conflict with Principle (A).

Under the framework of Definition \ref{defn:model}, we suggest to instead model the data generating process and sampling mechanism directly.
Theorem \ref{thm:main} establishes that every family of label equivariant finite sample models $\{P^{(n)}\}_{n\geq1}$ can be alternatively specified by a relatively exchangeable data generating model and a sampling mechanism, which together preserve the structure of $\Theta$.
Any model specified in this way is able to incorporate prior knowledge and beliefs about the key aspects of data generation.
%Inferences based on such a model are rational from the perspective of the data analyst.
%The issue of tractability is important, but computational concerns are relevant only after a satisfactory model has been identified.

\subsubsection{Sampling mechanisms}\label{section:sampling}
As important as the data generating model is the choice of sampling scheme, among which vertex, edge, and snowball sampling are the most widely discussed.
Lee, Kim \& Jeong \cite{LeeKimJeong2006} have studied the impact of vertex, edge, and snowball sampling schemes on observed network structure.
See \cite{AhmedNevilleKompella2010} for an overview of other sampling procedures.

Vertex and edge sampling proceed by a simple random sample of vertices and edges, respectively.
Snowball sampling is performed by iteratively expanding the neighborhood of a chosen vertex until the desired sample size is achieved. 
We discuss snowball sampling further in Section \ref{section:inference}.

Unlike snowball sampling,  vertex sampling does not resemble any realistic sampling scheme used in practice.
Vertex sampling is also logically indefensible in light of the sparsity hypothesis: a random sample of vertices from a sparse network is empty with high probability unless an appreciable fraction of vertices is sampled.  
%We are aware of no network dataset corresponding to an empty graph.

Random edge sampling is more tenable than vertex sampling in many cases, as many networks form by a process of interactions among vertices.
In this case, however, edges may be sampled in a size-biased manner according to the strength of ties between vertices.
%In Section \ref{section:examples-bick} we describe a new edge sampling scheme that is reasonable for certain datasets.

Theorem \ref{thm:main} establishes that even ill specified models can be recast as a statistical network model with a certain sampling scheme.
We give three such examples below.
%Our third example discusses a class of models for edge labeled networks introduced in \cite{CraneDempsey2016e2}.
% which directly addresses Question (II).
We discuss this further in Section \ref{section:power law}.

\paragraph{Bickel \& Chen's approach}

Bickel \& Chen \cite{BickelChen2009PNAS} propose a nonparametric approach based on the Aldous--Hoover theory of exchangeable vertex labeled graphs.
Let $h:[0,1]^2\to[0,1]$ be symmetric so that $G=(\Nb,E)$ satisfies
\begin{equation}\label{eq:Bick}
\mathbb{P}\{ij\in E\mid (U_n)_{n\geq1}\}=h(U_i,U_j),
\end{equation}
conditionally independently for all $i<j$, where $(U_n)_{n\geq1}$ are i.i.d.\ Uniform$[0,1]$ random variables.

Defining $\rho=\int_0^1\int_0^1 h(u,v)du dv$, Bickel \& Chen go on to write $w(u,v)=\rho^{-1}h(u,v)$ and assert \cite[p.\ 21069]{BickelChen2009PNAS} that 
\begin{quote}``it is natural finally to let $\rho$ depend on $n$ but $w(\cdot,\cdot)$ to be fixed.''\end{quote}
In light of Principle (A), we can think of nothing less natural.
Presumably, the approach is ``natural'' on the grounds that choosing, say, $\rho=n$ implies that the expected number of edges grows on the order of $n$ and the sequence of graphs is sparse in the sense of Definition \ref{defn:dense-sparse} below.
But without a way to relate distributions for different sample sizes, $h$ cannot be estimated in a way that is meaningful beyond the sampled network.
In such a case, the hypothetical property of sparsity, which tacitly assumes an infinite population of vertices, is moot.
See Section \ref{section:infer-univ} for a concrete example.

We can rectify this issue with a statistical network model $(\Theta,Q,\{\Sigma_n\}_{n\geq1})$ as follows.
The most obvious description takes $\Theta$ as the set of all symmetric functions $h:[0,1]^2\to[0,1]$ that are unique up to measure preserving transformations of $[0,1]^2$.
We then define $Q_{\theta}$ as in \eqref{eq:Bick} and $\{\Sigma_n\}_{n\geq1}$ as follows.
Let $(\rho_n)_{n\geq1}$ have $\rho_n\geq1$ for all $n\geq1$.
Given $G=(\Nb,E)$ and $n\in\Nb$, $\Sigma_n$ proceeds first by canonical sampling $G\mapsto\Rn G=G_n$ and then by thinning each of the edges of $G_n$ independently with probability $\rho_n^{-1}$ to obtain $\tilde{G}_n=([n],\tilde{E}_n)$:
\begin{equation}\label{eq:Bick-samp}
\mathbb{P}\{ ij\in\tilde{E}_n\mid G_n=([n],E_n)\}=\left\{\begin{array}{cc} \rho_n^{-1},& ij\in E_n,\\ 0,& \text{otherwise}.\end{array}\right.\end{equation}
The resulting finite sample models $\{Q^{(n)}\}_{n\geq1}$, though not consistent in the traditional sense, are related by the sampling scheme in \eqref{eq:Bick-samp}, permitting inference for $h$.
We stress that the above choice is just one of many, and the decision of which to use depends on which most accurately models reality.
We spell out the significance of this decision in Section \ref{section:infer-univ}.
%The appropriateness of the sampling scheme in \eqref{eq:Bick-samp} seems questionable for many applications.

\paragraph{Superstar model}

Bhamidi, et al \cite{BhamidiSteele2014} propose the {\em superstar model} for explaining the structure of networks based on Twitter activity.
The data are generated by retweet activity corresponding to specific events, such as the World Cup.
A {\em retweet} is a rebroadcasting of another user's activity, and in this network an edge between $v,v'\in V$ indicates that one of $v$ and $v'$ retweeted the other's tweet. 

A significant fraction of retweets corresponding to any given event tends to originate from a single individual, while the rest of the activity is spread across users.
Bhamidi, et al address this tendency by specifying a single vertex $v^*$, the {\em superstar}, and parameters $p\in(0,1)$ and $\delta>-1$.
A random network $G$ grows by sequential arrival of a new vertex at each time.
Upon arrival, the $(n+1)$st vertex either connects to $v^*$ with probability $p$ or with probability $1-p$ attaches to one of the other vertices $v$ with conditional probability proportional to $\deg(v)+\delta$, where $\deg(v)$ is the degree of vertex $v$.
These dynamics produce a tree with a single connected component with probability 1, but in general the network generated by retweet activity is neither a tree nor connected.
To fit their model, Bhamidi, et al process the data by first sampling the largest connected component and then removing edges so that the resulting dataset is a tree.

%These finite sample models $\{P^{(n)}\}_{n\geq1}$ are sampling consistent but not label equivariant.  
Altogether, the approach describes a family of sampling mechanisms $\{\Sigma_n\}_{n\geq1}$ and finite sample models $\{P^{(n)}\}_{n\geq1}$ but no data generating process and no way to fit the model to network data directly.
The model inflicts some uneasiness, as the sampling mechanism here was chosen for the purpose of forcing the data to fit the model rather than vice versa.
%We discuss this pitfall further in Section \ref{section:concluding remarks}.
We also point out that the above dynamics describe a network that grows by sequential addition of vertices when, in fact, retweet activity corresponds to a process of interactions among existing Twitter users.
In such a case, it is more appropriate to treat the edges as the units, as we do in the next model.

\paragraph{Edge exchangeable models}

Many network datasets are generated by a process of repeated interactions, as in the actors collaboration, Enron email, and retweet networks.
In all these cases, it is more appropriate to model the data generating process via edge addition, not vertex addition.
Furthermore, it is best to label the edges, instead of vertices, in order to better incorporate the fact that the network data are obtained by sampling the interactions, that is, movies, emails, or retweets.
These points have been discussed in detail throughout \cite{CraneDempsey2016e2}, which introduced and developed the class of edge exchangeable models for network data.

For a specific example of such a model, let $V$ be a countably infinite population.
To each $v\in V$, we assign a positive weight $W_v>0$ so that the ranked reordering $(W_v)_{v\in V}^{\downarrow}$ follows the Poisson--Dirichlet distribution with parameter $(\alpha,\theta)$ for $0<\alpha<1$ and $\theta>-\alpha$; see \cite{Crane2016ESF, FengPDbook} for further discussion of the many ways in which the Poisson--Dirichlet distribution arises.
Given $W$, we then generate a sequence of pairs $\{V_n,V_n'\}_{n\geq1}$, $V_n\neq V'_n$, as a conditionally i.i.d.\ sequence with 
\begin{equation}\label{eq:E2}\mathbb{P}\{\{V_n,V_n'\}= vv'\mid (W_x)_{x\in V}\}\propto W_vW_{v'},\quad v\neq v'\in V.\end{equation}

The sequence of pairs $\{V_n,V'_n\}_{n\geq1}$ determines a graph with labeled edges as in Figure \ref{fig:labeled}(d).
This graph is {\em edge exchangeable}, that is, invariant with respect to relabeling of the edges.
A more general construction of edge exchangeable network models can be described by taking $(W_v)_{v\in V}$ from any distribution on the infinite simplex, but we specialize here to the Poisson--Dirichlet case, which exhibits several relevant properties for Question (II).
%See \cite[Sections 3 and 4]{CraneDempsey2016e2} further details on the model in \eqref{eq:E2} and general theory of edge exchangeable networks.

\subsection{Sparsity, exchangeability, and projection sampling}\label{section:power law}

The edge exchangeable model given in \eqref{eq:E2}, and more generally in \cite{CraneDempsey2016e2}, exhibit several important properties which together offer a possible answer to Question (II) from Section \ref{section:summary}.
For one, edge exchangeability gives an invariance principle in accord with (A).
For two, this process generates a network with multiple edges, as when actors are cast together in more than one movie or individuals exchange multiple emails, and so the model more accurately reflects the nature of network formation than do models that describe growth by vertex addition.
From a network with multiple edges, we can obtain a network without multiple edges by projecting all multiple edges to a single edge, as in the karate club network \cite{Zachary1977}.
This {\em projection sampling} scheme is used to produce many network datasets, including those in \cite{KlimtYang2004,Zachary1977}.
For three, edge exchangeable models do allow for sparse and power law structure.
%The model in \eqref{eq:E2} behaves well under this operation.

\paragraph{Sparsity}
For any graph $G$ (vertex or edge labeled), we write $v(G)$ to denote the number of vertices and $e(G)$ to denote the number of edges in $G$.
For any $G\in\graphsn$, let
\[{\epsilon}(G):=\frac{2e(G)}{v(G)(v(G)-1)}\]
be the density of edges in $G$.

\begin{defn}[Sparse graphs]\label{defn:dense-sparse}
A graph $G\in\graphsN$ is {\em sparse} if 
\[\limsup_{n\rightarrow\infty}\epsilon(G_{|[n]})=0.\]
\end{defn}

\begin{defn}[Sparse network models]\label{defn:sparse model}
A network model $\mathcal{M}\subseteq\P(\graphsN)$ is {\em sparse} if $\mu$-almost every $G\in\graphsN$ is sparse for all $\mu\in\mathcal{M}$.
\end{defn}

Bickel \& Chen's approach from Section \ref{section:specification} seeks to model sparsity by specifying a family of finite sample models $P^{(n)}:\Theta\to\P(\graphsn)$ for which any sequence $(G_n)_{n\geq1}$ of finite graphs with $G_n\sim P^{(n)}_\theta$ has $\epsilon(G_n)\to_P0$ as $n\to\infty$, where $\to_P$ denotes {\em convergence in probability}.
This should not, however, be confused with a sparse network model as in Definition \ref{defn:sparse model}.
The finite sample models $\{P^{(n)}\}_{n\geq1}$ given in \eqref{eq:Bick-samp} are not sampling consistent and, therefore, do not determine a model $P:\Theta\to\P(\graphsN)$ for the population network.

Observation \ref{obs:dense or empty} dashes any hope of a sparse and exchangeable population generating model for countable vertex labeled networks.
The model in \eqref{eq:E2}, however, does give a straightforward generating mechanism which is both exchangeable and produces a sparse network with probability 1 under projection of multiple edges to a single edge.

Let $\Phi=\{(\alpha,\theta): 0<\alpha<1, \theta>-\alpha\}$ be the parameter space of the Poisson--Dirichlet distribution and let $Q:\Phi\to\P(\graphsN)$ be the model described by projecting multiple edges to a single edge in the multigraph generated by \eqref{eq:E2}.
For each $n\geq1$, let $\Rn$ be the canonical sampling mechanism on edge labeled graphs.
By \cite[Theorem 5.4]{CraneDempsey2016e2}, the model $(\Phi,Q,\{\Rn\}_{n\geq1})$ is sparse.

\paragraph{Power law exponent}

%Understanding power law behavior is another major objective of network analysis.
With the underlying network $G\in\graphsN$ understood, we write $N_{k,n}$ to denote the number of vertices with degree $k$ in $G_{|[n]}$.

\begin{defn}[Power law]\label{defn:power law}
A graph $G\in\graphsN$ exhibits {\em power law degree distribution with exponent $\gamma>1$} if 
\[ v(G_{|[n]})^{-1}N_{k,n}\rightarrow L(k) k^{-\gamma}\quad\text{as }n\rightarrow\infty\quad\text{for all large }k,\]
where $L(x)$ is a slowly varying function, that is, $\lim_{x \to \infty} L(tx) / L(x) = 1$ for all $t > 0$.  
\end{defn}

Power law distributions have been observed in various network datasets \cite{Abello1998,FFF1999,Kumar1999} and examined in broader scientific applications \cite{ClausetShaliziNewman2009}.
Understanding the power law phenomenon in networks is a matter of great interest and debate, with several authors questioning whether the power law reflects real network structure or is a sampling artifact \cite{AchlioptasClauset2005,KhaninWit2006,LeeKimJeong2006,WillingerAldersonDoyle2009}.

By properties of the Poisson--Dirichlet distribution  \cite[Chapter 3.3]{Pitman2005}, the proportion of vertices with degree $k$ in the multigraph generated by \eqref{eq:E2} with parameter $(\alpha,\theta)$ exhibits power law degree distribution with exponent $\alpha+1$.
Numerical observations also suggest that the power law is preserved upon projection of multiple edges to a single edge.
See the discussion in \cite[Sections 3 and 4]{CraneDempsey2016e2} for further details.

%\begin{example}[$\beta$-process]
%We discuss the $\beta$-process in the context of Theorem~\ref{thm:main} and discuss power law.
%The vertex-parameter sequence $\{ \beta_i \}_{i=1}^\infty$ is assumed exchangeable which yields an 
%exchangeable data generating model.  
%Fix~$[n]$, canonical sampling yields the process with $(\beta_1, \ldots, \beta_n)$
%and marginally~$p_{ij} = p \in [0,1]$ with expected degree~$p \cdot (n-1) = c$.
%Let ${\d} = \{ d_i \}_{i=1}^n$ denote an expected degree sequence that corresponds to a particular
%power law.  Then for each subset~$\beta$ of size~$n$ of the infinite sequence define
%\[
%\pi_{\tau,d} (\beta) \propto e^{\tau \cdot(\d^\prime \beta - m (\beta) ) } 
%\]
%where $m (\beta) = \sum_{1 \leq i < j \leq n} \log ( 1 + e^{\beta_i + \beta_j})$, and $\tau$ is 
%a concentration parameter.  We then sample $\beta$ according to this distribution.
%The observed graph will have expected power law behavior provided the the marginal 
%distribution of~$\beta_i$ has sufficient support.  
%\end{example}

\subsection{Relative exchangeability}\label{section:relative exchangeability}

For this section, we specialize to vertex labeled graphs.

Consider a network model $(\Theta,Q,\{\Rn\}_{n\geq1})$ and suppose $\Theta=\Phi\times\Psi$ decomposes into an exchangeable part $\Phi$ and a structural part $\Psi$ as in Section \ref{section:latent} above.
We assume $\Psi$ consists of countable relational structures, as defined in Appendix \ref{appendix:ultra}, which includes partitions, graphs, and other general structures in common use.
Without any further assumptions, Principle (A) suggests label equivariance of the data generating model $Q:\Theta\rightarrow\P(\graphsN)$.
To further ensure that the observed network data $G_n\in\graphsn$ is sufficient for inference of $\Theta_n=\Theta/\sim_{\Rn}$, we require relative exchangeability as in Definition \ref{defn:relative exchangeability}.
The next theorem characterizes relatively exchangeable network models.

\begin{thm}\label{thm:rel exch}
Suppose $\Theta$ decomposes as $\Phi\times\Psi$, for $\Psi$ satisfying Conditions \ref{assumption:TDC} and \ref{assumption:ultrahomogeneity}.
Let $(\Theta,Q,\{\Rn\}_{n\geq1})$ be an identifiable, relatively exchangeable statistical network model.
Then there exists a function $g:\Theta\times\tilde{\Psi}\times[0,1]^4\to\{0,1\}$ such that for $(\phi,\psi)\in\Phi\times\Psi$, $G\sim Q_{\phi,\psi}$ satisfies $G\equalinlaw G^*=(\Nb,E^*)$ for
\begin{equation}\label{eq:rel exch}
ij\in E^{*}\quad\text{if and only if}\quad g(\phi,\psi|_{[i\vee j]},U_0,U_i,U_j,U_{ij})=1,\quad j>i\geq1,\end{equation}
where $\{U_0, (U_i)_{i\geq1}, (U_{ij})_{j>i\geq1}\}$ are i.i.d.\ Uniform$[0,1]$ random variables and $\tilde{\Psi}=\bigcup_{n\geq1}\Psi/\sim_{\Rn}$ is the set of all equivalence classes of $\Psi$ under $\sim_{\Rn}$.
\end{thm}

\begin{rmk}
In \eqref{eq:rel exch}, $\psi|_{[i\vee j]}$ is the restriction of $\psi$ to its initial segment of units labeled $1,\ldots,i\vee j$.
In some cases, the representation in \eqref{eq:rel exch} can be simplified to only depend on $\psi|_{\{i,j\}}$, the restriction of $\psi$ to the units $i$ and $j$, but whether this is possible depends nontrivially on the structure of $\Psi$; see \cite{CraneTowsner2015}.
(The stochastic blockmodel admits the simpler representation, as we discuss in Example \ref{example:sbm-2}.)
\end{rmk}

\begin{rmk}
Theorem \ref{thm:rel exch} fits well with recent interest in the field of graphon estimation \cite{GaoLuZhou2015,WolfeOlhede2014}.  In Theorem \ref{thm:rel exch}, the function $g$ acts as a generalized graphon for relatively exchangeable models.  We discuss further in Section \ref{section:comm-det}.
\end{rmk}

\begin{example}[Stochastic blockmodel]\label{example:sbm-2}
The stochastic blockmodel from Example \ref{example:SBM} is relatively exchangeable and admits the following representation in terms of \eqref{eq:rel exch}.
In this case, $\Theta=\Phi\times\Psi$ for $\Phi=[0,1]\times[0,1]$ and $\Psi$ consists of all partitions of $\Nb$.
We identify $\Psi/\sim_{\Rn}$ with $\mathcal{P}_{[n]}$, the set of all partitions of $[n]$.
The model satisfies all the conditions of Theorem \ref{thm:rel exch} and can be represented by $g:\Phi\times \tilde{\Psi}\times[0,1]\rightarrow\{0,1\}$, where
\[g((p,q),B|_{[i\vee j]},U_{ij})=\mathbf{1}\{U_{ij}\leq p\}\mathbf{1}\{i\approx_B j\}+\mathbf{1}\{U_{ij}\leq q\}\mathbf{1}\{i\not\approx_B j\}.\]
\end{example}

We acknowledge here that Theorem \ref{thm:rel exch} characterizes a large class of relatively exchangeable network models relevant for practical purposes.
Conditions \ref{assumption:TDC} and \ref{assumption:ultrahomogeneity}, though abstract, are reasonable for most practical purposes.

\section{Inference}\label{section:inference}
Our proposed framework does not favor any inferential method or statistical philosophy over another, but it does affect how estimates can be interpreted and to what extent statistical models are useful for inferences beyond the sample.
Sections \ref{section:infer-univ} and \ref{section:missing} emphasize that valid inference of universal parameters and predictive probabilities relies on an accurate model for the sampling mechanism.
Section \ref{section:comm-det} foreshadows future developments in the realm of generalized graphon estimation.

\subsection{Inferring universal parameters}\label{section:infer-univ}

Estimation of universal parameters from subnetwork data requires complete identifiability as well as knowledge of the sampling mechanism by which the parameters in the population and finite sample models are related.
Ideally, the parameters maintain their meaning under sampling, but Observation \ref{obs:dense or empty} and the preceding discussion demonstrates how this fails in many applications.

Consider the following special case of Bickel \& Chen's model from Section \ref{section:sampling}.
Let $\Theta=[0,1]$ and $P^{(n)}:\Theta\to\P(\graphsn)$ be defined so that, for each $\theta\in[0,1]$, edges are present in $G=([n],E)$ independently with probability
\begin{equation}\label{eq:fidi-Bick}P^{(n)}_{\theta}\{ij\in E\}=\theta/n,\quad 1\leq i<j\leq n.\end{equation}
These finite sample models $\{P^{(n)}\}_{n\geq1}$ are not sampling consistent and, therefore, they do not directly correspond to a data generating model.
There are, however, innumerably many ways to fit this model into our framework.
We discuss a few to highlight the salience of our discussion.

In all cases, we let $\Sigma_n$ be the random sampling scheme in \eqref{eq:Bick-samp} with $\rho_n=n$ and we write $\hat{p}_n=\binom{n}{2}^{-1}e(\tilde{G}_n)$ for the proportion of edges in the sampled data $\tilde{G}_n=([n],\tilde{E})$ with $n$ labeled vertices.
We compute the likelihood by
\[\mathcal{L}(\theta\mid\tilde{G}_n)\propto\prod_{1\leq i<j\leq n}(\theta/n)^{\mathbf{1}\{ij\in \tilde{E}\}}(1-\theta/n)^{1-\mathbf{1}\{ij\in \tilde{E}\}},\quad \tilde{G}_n=([n],\tilde{E})\in\graphsn,\]
from which we obtain the maximum likelihood estimator $\hat{\theta}_{MLE}^{(n)}=n\hat{p}_n\wedge 1$.

Note that this is the same estimator we would obtain if we only assume the finite sample models $\{P^{(n)}\}_{n\geq1}$ in \eqref{eq:fidi-Bick} and proceed by maximum likelihood estimation.
A key difference, however, is that our estimator $\hat{\theta}^{(n)}_{MLE}$ is logically connected to the population parameter $\theta$, while the same estimate obtained from the inconsistent collection of finite sample models is not.
The estimators from $\{P^{(n)}\}_{n\geq1}$, without the corresponding data generating model and sampling mechanisms, though symbolically identical to $\hat{\theta}^{(n)}_{MLE}$ for each $n\geq1$, establish no identical since there is no connection between the parameters `$\theta$' for different sample sizes.

This is not semantics.  
Consider the same finite sample models $\{P^{(n)}\}_{n\geq1}$ above but now suppose that $Q:\Theta\to\P(\graphsN)$ defines $Q_{\theta}$ as the Erd\H{o}s--R\'enyi model with parameter $\theta/(2-\theta)$.
Notice that $x\mapsto x/(2-x)$ is a bijection of $[0,1]$ so that $Q\Theta$ is the same set of distributions, and thus determines the same model as above, but with a different interpretation given to $\Theta$.
Under the sampling scheme in \eqref{eq:Bick-samp} with $\rho_n=n$, $Q$ induces the same finite sample models as $\{P^{(n)}\Theta\}_{n\geq1}$, but a different interpretation for $\theta$.
We settle on the estimator $\tilde{\theta}^{(n)}=(n\hat{p}_n\wedge1)/(1+n\hat{p}_n\wedge1)$, which differs from the maximum likelihood estimator of $\hat{\theta}^{(n)}_{MLE}=n\hat{p}_n\wedge1$ when the finite sample models $\{P^{(n)}\}_{n\geq1}$ are considered in isolation.
Notice that $\tilde{\theta}^{(n)}\to_P\theta$ whereas $\hat{\theta}^{(n)}_{MLE}=n\hat{p}_n\to_P\theta/(2-\theta)$.
Thus, even in this simple setting, the finite sample models are correct, but without the correct logical connection to the data generating process, via the sampling mechanism, the parameter assumes a different meaning in the finite sample models and population model.

The above example applies to any continuous bijection $f:[0,1]\to[0,1]$: if $Q:\Theta\to\P(\graphsN)$ defines $Q_{\theta}$ as the Erd\H{o}s--R\'enyi model with parameter $f(\theta)$ and $Q^{(n)}$ are the finite sample models induced by sampling as in \eqref{eq:Bick-samp}, then $\tilde{\theta}^{(n)}=f^{-1}(n\hat{p}_n\wedge1)$ is a consistent estimator of $\theta$ and $\theta^{(n)}_{MLE}=n\hat{p}_n\wedge1$ is not.
Though perhaps obvious how to resolve the issue in this simple case, it is generally quite difficult for more sophisticated models, such as the exponential random graph model \cite{RinaldoShalizi2013}.

\subsection{Missing link prediction}\label{section:missing}
%We illustrate the effect of sampling on missing link prediction in the simple setting of a network with 3 vertices.
For known $p\in(0,1)$, let $G$ be modeled by the Erd\H{o}s--R\'enyi distribution with parameter $p$ on graphs with 3 vertices.
We consider the predictive probability of an edge between vertices labeled 1 and 3 in the event that two edges $\{1,2\}$ and $\{2,3\}$ were sampled.
We stress that, in general, the observation $(\{1,2\},\{2,3\})$ conveys different information than $(\{1,2\},\{1,3\})$ because the labels assigned to vertices during sampling need not be exchangeable.
The following analysis, though carried out in the simple case of 3 vertices, illustrates the impact of sampling scheme on link prediction.
\begin{itemize}
	\item {\em Vertex sampling}.  We assume the 3 vertices are labeled uniformly without replacement and we are given the information that edges $\{1,2\}$ and $\{2,3\}$ are present, but no information as to the presence or absence of $\{1,3\}$ in the population network.  Since $p$ is known and all edges behave independently, the predictive probability that edge $\{1,3\}$ is present is $p$.
	\item {\em Edge sampling}.  We assume 2 edges are sampled uniformly with replacement among the edges in the underlying graph.
	The vertices $v_1$ and $v_2$ in the first sampled edge are labeled uniformly without replacement in $\{1,2\}$ and the unsampled vertex is assigned label 3.  After 2 edges are sampled, the possible observations are $(\{1,2\},\{1,2\})$, $(\{1,2\},\{1,3\})$, and $(\{1,2\},\{2,3\})$.  Given observation $(\{1,2\},\{2,3\})$, the edge $\{1,3\}$ is present in the underlying graph with probability $4p/(9-5p)$.
	\item {\em Snowball sampling}.  Under snowball sampling, we start with a vertex $v_1$ chosen uniformly among the three vertices.
	We assign this vertex the label 1 and then sample outwardly by choosing $v_2$ uniformly among the vertices adjacent to $v_1$.
	We assign label $2$ to $v_2$ and then sample outwardly from $v_2$ to one of its neighbors $v_3\neq v_1$ and assign label 3 to $v_3$.
	If, at any point, there are no vertices to choose from, we choose the next vertex uniformly among the unsampled vertices.
	Given $(\{1,2\},\{2,3\})$, the probability of the edge $\{1,3\}$ is $p/(2-p)$.
	\item {\em Bickel \& Chen's model}.  Under Bickel \& Chen's approach, we observe a thinned version of the Erd\H{o}s--R\'enyi graph with parameter $p$, which maintains each edge independently with probability $1/3$.  In this case, the predictive probability that $\{1,3\}$ is present, given we observe $(\{1,3\},\{2,3\})$ and no edge $\{1,3\}$, is $2p/(3-p)$.	
	\end{itemize}

\subsection{Community detection}\label{section:comm-det}

The characterization of relatively exchangeable models in Theorem \ref{thm:rel exch} suggests a general approach to community detection with ties to recent trends in nonparametric graphon estimation.
In a nonparametric setting, we assume the parameter space $\Theta=\Psi$, where $\Psi$ is the set of all partitions of $\Nb$.
Given $\psi\in\Psi$, we model the network data $G^*=(\Nb,E^*)$ by 
\[\mathbb{P}\{ij\in E^*\mid(U_k)_{k\geq1}\}= g(\mathbf{1}\{i\approx_{\psi} j\},U_i,U_j)\]
conditionally independently for all $j>i\geq1$, where $(U_k)_{k\geq1}$ are i.i.d.\ Uniform$[0,1]$ random variables, $\mathbf{1}\{i\approx_{\psi} j\}$ is the indicator of the event that $i$ and $j$ are in the same block of $\psi$, and $g:\{0,1\}\times[0,1]^2\to[0,1]$ is symmetric in its last two arguments.
This setup generalizes the stochastic blockmodel in Example \ref{example:SBM}, which corresponds to 
\[g(\mathbf{1}\{i\approx_{\psi} j\},U_i,U_j)=\left\{\begin{array}{cc} p,& i\approx_{\psi} j,\\ q,& \text{otherwise.}\end{array}\right.\]

An even more general setting is possible based on Theorem \ref{thm:rel exch} in which the first argument of $g$ depends on the restriction $\psi|_{[i\vee j]}$, that is, the entire partition $\psi$ induces on $1,\ldots,i\vee j$.
Given the current interest in graphon estimation, the setup here seems a natural class of nonparametric models for community detection.

\section{Concluding remarks}\label{section:concluding remarks}

The preceding pages lay a foundation for the development of sound statistical theory and methods for network data.
Though formal in spots, the conversation is rooted in practical concerns about network modeling.
Given the technical nature of the discussion, we conclude with some parting shots directed toward applied statisticians who work with network data.

The preceding discussion demonstrates the dangers inherent in the standard protocol for analyzing network data.
The logical fallacy of modeling data with inconsistent finite sample distributions is well understood by statisticians, and yet the practice endures throughout the statistics literature on networks.
The wide acceptance of this otherwise unacceptable practice is a triumph of pragmatism over principle.
Absent the foregoing framework, the analyst must choose between throwing up his hands and doing nothing or performing an analysis which, though not iron clad, provides some useful insights.

Our primary observations, summarized in points (M1)-(M4), highlight several important facets of network modeling that have otherwise gone unnoticed or unspoken.
Most importantly, our suggested framework calls for explicit models for both the data generating process and the sampling mechanism.
We stress that both of these components correspond to a physical process and, therefore, the choice of each should reflect the analyst's best knowledge about the real world.
To wit, neither $Q:\Theta\to\P(\graphsN)$ nor $\{\Sigma_n\}_{n\geq1}$ should be chosen solely because inference is tractable or computationally efficient under a given selection.
In light of concerns over the incompatibility between common invariance principles and empirical properties, we expand the suite of viable network models to include both edge exchangeable, relationally exchangeable, and relatively exchangeable data generating processes.

The pragmatist will undoubtedly raise the concern that a model specified this way does not generally yield closed form expressions for the finite sample models.
One may inquire, however, as to the benefit of closed form finite sample models that are incompletely specified and ineffectual for out-of-sample inference.
For example, it is not clear how to recover the finite sample models of the exponential random graph model (ERGM) from an exchangeable data generating model and a reasonable sampling scheme.
But if the finite sample models do not correspond to some realistic data generating model and a reasonable sampling scheme, then under what circumstances is it logically valid or defensible to model network data with the ERGM?
The discussion of Sections \ref{section:infer-univ} and \ref{section:missing} demonstrates that these concerns are in no way specific to the ERGM.  We caution against the general practice of fitting network data to a poorly understood model on the grounds of practical expediency.
Ultimately, the goal of sensible inference can only be achieved if the chosen model accurately depicts reality.

\appendix{\section{Conditions for main theorems}\label{appendix:ultra}}

Theorem \ref{thm:main} establishes that every label equivariant model for network data entails an exchangeable, identifiable data generating process and a sampling mechanism.
To ensure identifiability we require Condition \ref{condition:uncountable}.

Theorem \ref{thm:rel exch} characterizes relatively exchangeable network models for the case when the parameter space consists of relational structures that satisfy Conditions \ref{assumption:TDC} and \ref{assumption:ultrahomogeneity} below.
Certain ideas from Theorem \ref{thm:rel exch} enter into our proof of Theorem \ref{thm:main}.

\subsection{Identifiability}\label{condition:uncountable}
The parameter space $\Theta$ decomposes as $\Phi\times\Psi$ for an exchangeable part $\Phi$ which maps injectively into $(0,1)$ and a combinatorial part $\Psi$ which corresponds to some class of countable relational structures with finite signature.  
In general $\Psi$ consists of structures $\psi=(\Nb; R_1^{\psi},\ldots,R_r^{\psi})$ such that $R_j^{\psi}\subseteq\Nb^{i_j}$ for some $i_j\geq1$, for each $j=1,\ldots,r$.
Such structures include graphs and partitions, with $\psi=(\Nb; R_1^{\psi})$ having $R_1^{\psi}\subseteq\Nb\times\Nb$ in each case.
See \cite{CraneTowsner2015} for specifics.

\subsection{Strong amalgamation property}\label{assumption:TDC}
We require that every $\psi\in\Psi$ has the {\em strong amalgamation property}.
Regarding $\psi\in\Psi$ as a combinatorial structure labeled by $\Nb$, we write $\psi|_{S}$ to denote the restriction of $\psi$ to the structure labeled by $S\subset\Nb$. 
The strong amalgamation property says that if $\psi|_{S}$ embeds into two structures $\psi_1$ and $\psi_2$ in $\Psi$, then there is a common structure $\psi^*\in\Psi$ into which both $\psi_1$ and $\psi_2$ embed such that the only elements of $\psi_1$ and $\psi_2$ that are identified are those which are already identified by the embeddings of $\psi_{|S}$ into $\psi_1$ and $\psi_2$, respectively.
See \cite[Section 6.4]{HodgesModelTheory} for further details.
Most structures found in practice satisfy the strong amalgamation property, including partitions and graphs.
%Together with the ultrahomogeneity property, strong amalgamation is necessary and sufficient for the existence of exchangeable measures on isomorphism classes of certain structures \cite{AFP}.

\subsection{Ultrahomogeneity}\label{assumption:ultrahomogeneity}

A combinatorial structure $\psi\in\Psi$ is {\em ultrahomogeneous} if for every embedding of a finite structure into $\psi$ extends to an automorphism of $\psi$. 
Ultrahomogeneity of $\psi\in\Psi$ ensures that we can define a relatively exchangeable model $Q:\Phi\times\Psi\to\P(\graphsN)$ in keeping with the lack of interference condition.
We assume $\Psi$ contains an ultrahomogeneous $\psi\in\Psi$ such that every $\psi'\in\Psi$ embeds into $\psi$.

\section{Proofs of theorems}\label{appendix:proofs}

\subsection{Proof of Observation \ref{obs:dense or empty}}
The empty graph is clearly exchangeable, so we focus on the case when $G=(\Nb,E)$ is exchangeable but not almost surely empty.
We need to show that the limiting edge density $\epsilon(G)$ is strictly positive with probability 1.
For each $n\geq1$, let $X_n:=\epsilon(G_{|[n]})$ be the edge density of $G_{|[n]}$.
By exchangeability,
\[\mathbb{E}(X_n\mid X_{n+1})=X_{n+1}\quad\text{for all }n\geq1,\]
so that $(X_n)_{n\geq1}$ is a reverse martingale and has an almost sure limit $X_{\infty}$.

If $X_{\infty}>0$, then $G$ is dense by definition.
If $X_{\infty}=0$, the bounded convergence theorem and exchangeability imply
\[\mathbb{E}(\lim_{n\rightarrow\infty}X_n)=\lim_{n\rightarrow\infty}\frac{2}{n(n-1)}\sum_{1\leq i<j\leq n}\mathbb{E}(\mathbf{1}\{ij\in E\})=0,\]
so that $G$ is empty with probability 1.

\subsection{Proof of Theorem \ref{thm:main}}

We call a graph $G\in\graphsN$ {\em universal} if for every finite subgraph $G'\in\graphsn$ there exists $S\subset\Nb$ with $|S|=n$ such that $G_{|S}\cong G'$.
In other words, every finite subgraph is embedded in $G$.

\begin{lemma}\label{lemma:Rado}
Let $\mu_{\theta}$ denote the Erd\H{o}s--R\'enyi measure with parameter $0<\theta<1$ on the space $\graphsN$ with vertices labeled in $\Nb$.
For $0<\theta<1$, $G\sim\mu_{\theta}$ is universal and ultrahomogeneous with probability 1.
\end{lemma}

\begin{proof}
For $n\geq1$, let $\mu_{\theta}^{(n)}=\mu_{\theta}\Rninv$ be the measure induced on $\graphsn$.
For all $0<\theta<1$ and all $F\in\graphsn$, $\mu_{\theta}^{(n)}(F)>0$.
For $G\sim\mu_{\theta}$ and $F\in\graphsm$, let $E_k=\{G_{|\{mk+1,\ldots,m(k+1)\}}\cong F\}$ be the event that the subgraph $G_{|\{mk+1,\ldots,m(k+1)\}}$ coincides with $F$, for $k=0,1,\ldots$.
In the Erd\H{o}s--R\'enyi model, all edges are present or absent independently with probability $\theta$, implying $\mu_{\theta}(E_k)=\mu_{\theta}^{(m)}(F)>0$ for every $k\geq0$, $\{E_k\}_{k\geq0}$ are independent, and $\sum_{k\geq0}\mu_{\theta}(E_k)=\infty$.
The second Borel--Cantelli lemma implies that there are infinitely many copies of $F$ in $G$ with probability 1.
Since the set $\bigcup_{n\geq1}\graphsn$ of finite subgraphs is countable, it follows that every finite subgraph occurs in $G$ with probability 1 and $G\sim\mu_{\theta}$ is universal almost surely.

Ultrahomogeneity follows by extending the above argument.
Writing $G_{ij}=\mathbf{1}\{ij\in E\}$ to indicate $ij\in E$ for $G=(\Nb, E)$, we suppose that $(G_{s_is_j})_{1\leq i,j\leq m}=(F_{ij})_{1\leq i,j\leq m}$ for some ordered subset $(s_1,\ldots,s_m)$ of distinct labels.
We extend $(s_1,\ldots,s_m)$ beginning at $s^*=1+\max(s_1,\ldots,s_m)$ and choosing $s_{m+1}$ to be the smallest integer such that $(G_{s_is_j})_{1\leq i,j\leq m+1}=(F_{ij})_{1\leq i,j\leq m+1}$.
This event requires only that the finite sequences $(G_{s_1s_{m+1}},\ldots,G_{s_ms_{m+1}})$ and $(F_{1,m+1},\ldots,F_{m,m+1})$ coincide.
For each choice of $s_{m+1}$, this happens independently with probability at least $\min\{\theta^n,(1-\theta)^n\}>0$.
Borel--Cantelli again implies that there is such an $s_{m+1}$ with probability 1.
\end{proof}

\begin{lemma}
Let $\mu_{\alpha}$ be the edge exchangeable probability measure driven by the Poisson--Dirichlet distribution with parameter $(\alpha,1)$, $0<\alpha<1$, on the space $\graphsN$ with edges labeled in $\Nb$ as in Sections \ref{section:specification} and \ref{section:power law}.
For $0<\alpha<1$, $G\sim\mu_{\alpha}$ is universal and ultrahomogeneous with probability 1.
\end{lemma}

\begin{proof}
The proof is identical {\em mutatis mutandis} to that of Lemma \ref{lemma:Rado} upon realizing that $\mu_{\alpha}\Rninv(F)>0$ for every finite edge labeled graph $F\in\graphsn$, where $\mu_{\alpha}\Rninv$ is the measure $\mu_{\alpha}$ induces on graphs with edges labeled $1,\ldots,n$ by canonical sampling.
\end{proof}
\begin{proof}[Proof of Theorem \ref{thm:main}]
For definiteness, we can take $\mu$ to be the Erd\H{o}s--R\'enyi measure with success probability $p\in(0,1)$.
By Lemma \ref{lemma:Rado}, $\mu$-almost every $G\in\graphsN$ is universal and ultrahomogeneous.
Any probability distribution $\mu_n$ on $\graphsn$ induces distribution $\mu_{n,m}$ on $\graphsm$, $m\leq n$, by subsampling, $\mu_{m,n}:=\mu_n\Rmninv$.
We define $\Sigma_n:\graphsN\rightarrow\graphsn$ by the following random sampling mechanism.

Beginning with $s_1=1$ and the unique graph on one vertex $\Gamma_1=G_{|[1]}$, we build $(\Gamma_1,\ldots,\Gamma_n)$ sequentially so that $\Gamma_m\sim\mu_{n,m}$ for every $m=1,\ldots,n$.  In particular, $\Gamma_n\sim\mu_n$ as desired.
At stage $m$, suppose we have $\Gamma_m=G_{|\{s_1,\ldots,s_m\}}$.
We sample $s_{m+1}>s_m$ such that $\Gamma=G_{|\{s_1,\ldots,s_{m+1}\}}$ has distribution $\mu_{n,m+1}$ by drawing $G^*\in\mathcal{G}_{[m+1]}$ according to
\[\mathbb{P}\{G^*=F\}=\left\{\begin{array}{cc}
\frac{\mu_{n,m+1}(F)}{\mu_{n,m}(G_{|\{s_1,\ldots,s_m\}})},& F_{|[m]}=G_{|\{s_1,\ldots,s_m\}},\\
0,& \text{otherwise,}
\end{array}\right.\]
for each $F\in\mathcal{G}_{[m+1]}$.
We then choose $s_{m+1}>s_m$ to be the smallest value such that $G_{\{s_1,\ldots,s_{m+1}\}}=F$.
Existence of such an $s_{m+1}$ follows by Lemma \ref{lemma:Rado}.
The fact that $\Gamma_{m+1}\sim\mu_{n,m+1}$ follows by exchangeability of the Erd\H{o}s--R\'enyi process.

We complete the proof by letting $t:\Theta\to(0,1)$ be any injection as guaranteed by Condition \ref{condition:uncountable}.
We then define $Q:\Theta\rightarrow\P(\graphsN)$ by taking $Q_{\theta}$ to be the Erd\H{o}s--R\'enyi distribution with parameter $t(\theta)$ for each $\theta\in\Theta$ and defining $\Sigma_n:\graphsN\rightarrow\graphsn$ to be the sampling mechanism defined implicitly above.

The above argument follows through if we instead work with edge labeled graphs and sample edges instead of vertices.
This completes the proof.
\end{proof}

\begin{rmk}
Our proof of Theorem \ref{thm:main} does not imply that $Q:\Theta\rightarrow\P(\graphsN)$ must correspond to the Erd\H{o}s--R\'enyi or Poisson--Dirichlet model.
For a given collection of finite sample models $\{P^{(n)}:\Theta\rightarrow\P(\graphsn)\}_{n\geq1}$, there may be infinitely many choices of data generating model $Q:\Theta\rightarrow\P(\graphsN)$ and sampling mechanism $\{\Sigma_n\}_{n\geq1}$.
Finding the appropriate combination is the art of statistical modeling.
\end{rmk}

\subsection{Proof of Theorem \ref{thm:rel exch}}

Under Conditions \ref{condition:uncountable} and \ref{assumption:TDC}, there exists an ultrahomogeneous $\psi^*\in\Psi$ such that every $(\phi,\psi)$ embeds into $(\phi,\psi^*)$ in the sense that there is an injection $\pi:\Nb\to\Nb$ such that $G\sim Q_{\phi,\psi^*}$ implies $G^{\pi}\sim Q_{\phi,\psi}$, where $G^{\pi}=(\Nb, E^{\pi})$ is defined by
\[(i,j)\in E^{\pi}\quad\text{if and only if}\quad(\pi(i),\pi(j))\in E.\]
The proof is completed by noticing that the conditions of \cite[Theorem 3.15]{CraneTowsner2015} apply to $(\phi,\psi^*)$, giving the representation of every distribution in $Q$ by \eqref{eq:rel exch}.

\bibliography{network-refs}
\bibliographystyle{abbrv}

\end{document}